\documentclass[12pt]{amsart}
\usepackage{amsfonts}
\usepackage{amssymb}

\usepackage{amsmath}
\usepackage{epsf}
\usepackage{epsfig}
\usepackage{colordvi}
\usepackage{fancybox}

\newcommand{\R}{\mathbb R}

\newtheorem{thm}{Theorem}[section]

\newtheorem{lem}[thm]{Lemma}
\newtheorem{prop}[thm]{Proposition}
\theoremstyle{definition}
\newtheorem{defn}[thm]{Definition}

\theoremstyle{remark}
\newtheorem{remark}{Remark}[section]

\newtheorem{example}{Example}[section]

\newcommand{\ds}{\displaystyle}

\begin{document}

\title[Canonical parameters on a surface in $\mathbb R^4$  ]
{Canonical parameters on a surface in $\mathbb R^4$}%

\author{Ognian Kassabov and Velichka Milousheva}%
\address{Institute of Mathematics and Informatics, Bulgarian Academy of Sciences,
Acad. G. Bonchev Str. bl. 8, 1113, Sofia, Bulgaria}
\email{okassabov@abv.bg}
\email{vmil@math.bas.bg}


\begin{abstract}

In the present paper, we study surfaces in the four-dimensional Euclidean space $\mathbb{R}^4$. We define special principal parameters, which we call canonical, on each surface without minimal points, and prove that the surface admits (at least locally) canonical principal parameters. They can be considered as a generalization of the canonical parameters for minimal surfaces and the canonical parameters for surfaces with parallel normalized mean curvature vector field introduced before. We prove a fundamental existence and uniqueness theorem formulated in terms of canonical principal parameters, which states that the surfaces in $\mathbb{R}^4$  are determined up to a motion by four geometrically determined functions satisfying a system of partial differential equations.

\end{abstract}

\maketitle

%============================================================================

%   Introduction

%=============================================================================

\section{Introduction}

A fundamental problem in the local theory of surfaces both in Euclidean and pseudo-Euclidean spaces is the problem of finding a minimal number of functions, satisfying some natural conditions (differential equations), that determine the surface up to a motion in the ambient space. This problem is  solved for minimal surfaces (surfaces with zero mean curvature) of co-dimension two in the Euclidean 4-space $\R^4$. In \cite{Itoh}, T. Itoh introduced special geometric parameters on any minimal surface in  $\R^4$  and later they were used by R. Tribuzy and I. Guadalupe to prove that a minimal surface in  $\R^4$  is determined up to a motion by two invariant functions (the Gaussian curvature and the normal curvature) satisfying a system of two PDEs \cite{Trib-Guad}. The same problem was solved for spacelike and timelike zero mean curvature surfaces in the Minkowski 4-space $\R^4_1$ in \cite{Al-Pal} and \cite{G-M-IJM}, respectively. It was proved that the spacelike or timelike surfaces with zero mean curvature in  $\R^4_1$  admit (at least locally) special isothermal parameters, called \textit{canonical},  such that the two main invariants -- the Gaussian curvature and the normal curvature of the surface satisfy a system of two partial differential equations (called a \textit{system of natural PDEs}). In all these cases, the number of the invariant functions determining the surfaces and the number of the differential equations are both reduced to two. 

Another class of surfaces,  different from the minimal ones, for which the number of functions and the number of differential equations can be reduced, is the class of surfaces with parallel normalized mean curvature vector field. In \cite{G-M-Fil}, G. Ganchev and the second author proved that the surfaces with parallel normalized mean curvature vector field (PNMCVF) in the Euclidean 4-space $\R^4$ and the spacelike PNMCVF-surfaces in the Minkowski 4-space $\R^4_1$ can be described  in terms of three  functions satisfying a system of three partial differential equations. 
A similar result is proved for the class of  timelike surfaces with parallel normalized mean curvature vector field in the Minkowski space $\R^4_1$, see \cite{B-M}. 
The basic  approach to the study of PNMCVF-surfaces both in Euclidean and pseudo-Euclidean spaces  is based on introducing of special geometric parameters on each  such surface which are called \textit{canonical parameters}.  In the case of spacelike PNMCVF-surfaces, canonical principal parameters are introduced, while in the case of timelike PNMCVF-surfaces in $\R^4_1$, canonical isotropic parameters are used. 

\vskip 1mm
In the present paper, we consider the general class of surfaces in the Euclidean 4-space $\R^4$. First we define  canonical principal parameters for these surfaces and then we prove that any surface in $\R^4$ admits (at least locally) canonical parameters. They generalize the canonical parameters introduced for minimal surfaces and for PNMCVF-surfaces and allow us to describe the general class of surfaces in terms of four functions. The main result in the paper is  the Fundamental existence theorem (Theorem \ref{eq:MainTheorem}) formulated in terms of canonical parameters, which states that any surface in $\R^4$  is determined up to a motion by four geometrically determined functions.
Finally, we give some examples that demonstrate our theory.

\section{Preliminaries}

We consider surfaces in the 4-dimensional Euclidean space  $\R^4$  endowed with the standard flat metric 
$\langle ., . \rangle$. If $M^2$ is a  2-dimensional surface in $\R^4$, then $\langle ., . \rangle$ induces a Riemannian metric $g$ on $M^2$. Denote by  $\nabla'$ and $\nabla$ the Levi Civita connections on $\R^4$ and $M^2$,  respectively. We use the notation $\sigma$ for the second fundamental tensor of the surface, i.e. 
$$\nabla'_xy = \nabla_xy + \sigma(x,y)$$
for any tangent vector fields $x$ and $y$.

Let $M^2: z = z(u,v), \, \, (u,v) \in {\mathcal D}$ (${\mathcal D} \subset \R^2$) be a local parametrization of a 2-dimensional surface in $\R^4$ and denote by $E$, $F$, $G$ the coefficients of the 
first fundamental form of $M^2$. Let $\{n_1,n_2\}$ be an orthonormal frame field of the normal bundle. Then, we have the following derivative formulas
$$
	\begin{array}{l}
	\vspace{2mm}
		z_{uu}=\Gamma_{11}^1z_u+\Gamma_{11}^2z_v +c_{11}^1n_1+c_{11}^2n_2,  \\
		\vspace{2mm}
		z_{uv}=\Gamma_{12}^1z_u+\Gamma_{12}^2z_v +c_{12}^1n_1+c_{12}^2n_2, \\
		\vspace{2mm}
		z_{vv}=\Gamma_{22}^1z_u+\Gamma_{22}^2z_v +c_{22}^1n_1+c_{22}^2n_2,
	\end{array}
$$ 
$\Gamma_{ij}^k$ being the Christofel's symbols and $c_{ij}^k$, $i, j, k = 1,2$ being smooth functions on $M^2$. The coefficients of the second fundamental form
of $M^2$ in $\mathbb R^4$, defined in \cite{G-M-2008, G-M-2010}, are given by 
$$
	L\!=\!\frac{2}{\sqrt{EG-F^2}} \left|\begin{array}{cc} 
	\vspace{2mm}
	          c_{11}^1 & c_{11}^2 \\
						c_{12}^1 & c_{12}^2
					\end{array}\right|;
   \;\;
	M\!=\!\frac{1}{\sqrt{EG-F^2}} \left|\begin{array}{cc} 
	\vspace{2mm}
	          c_{11}^1 & c_{11}^2 \\
						c_{22}^1 & c_{22}^2
					\end{array}\right|;
	\;\;
	N\!=\!\frac{2}{\sqrt{EG-F^2}}\left|\begin{array}{cc} 
	\vspace{2mm}
	          c_{12}^1 & c_{12}^2 \\
						c_{22}^1 & c_{22}^2
					\end{array}\right|.
$$
It is proved in \cite{G-M-2008} that $L=M=N=0$ if and only if $M^2$ is contained in
a hyperplane of $\mathbb R^4$  or $M^2$ is a developable ruled surface. 

A Weingarten-type linear map $\gamma$ is defined in the tangent space at any point of the surface, see \cite{G-M-2008, G-M-2010}. This map 
generates two invariant functions $k$ and $\varkappa$ which are related to the coefficients
of the first and the second fundamental forms by the formulas
$$
	k={\rm det}\, \gamma=\frac{LN-M^2}{EG-F^2} \,, \qquad \varkappa=-\frac12 {\rm trace}\, \gamma=\frac{EN-2FM+GL}{2(EG-F^2)} \ .
$$
Since $\gamma$ is a symmetric linear operator, we have $\varkappa^2-k\ge0$ and $\varkappa^2-k=0$ if and only if  (see \cite{G-M-2008}) 
$$
	L=\rho E\,, \qquad M=\rho F\,, \qquad N=\rho G  \qquad (\rho \in \mathbb R) .
$$
It is proved also that the invariant $\varkappa$ coincides with the curvature of the normal connection (normal curvature) of the surface. 

The minimal surfaces (surfaces whose mean curvature vector vanishes at each point) are characterized by the equality $\varkappa^2-k=0$ at each point \cite{G-M-2008}. A result of Eisenhart \cite{Eisen} shows that the class of minimal super-conformal\footnote{Surfaces whose ellipse of normal curvature at each point is a circle.} surfaces in $\R^4$ is locally equivalent to the class of holomorphic curves in  $\mathbb{C}^2 \equiv \R^4$. The Gauss curvature $K$ and the curvature of the normal connection $\varkappa$ of any minimal non-super-conformal surface parametrized by special isothermal parameters satisfy
the following system of  partial differential equations \cite{Trib-Guad}:
\begin{equation} \label{e:sys-min}
\begin{array}{l}
\ds{(K^2 - \varkappa^2)^{\frac{1}{4}}\, \Delta \ln (K^2 - \varkappa^2)} = 8 K\\
[2mm]
\ds{(K^2 - \varkappa^2)^{\frac{1}{4}}\, \Delta \ln \frac{K-\varkappa}{K + \varkappa}= -4 \varkappa}
\end{array}
\end{equation} 
where $\Delta = \ds{\frac{\partial^2}{\partial u^2} + \frac{\partial^2}{\partial v^2}}$ is the Laplace operator.

Further, we consider surfaces in $\R^4$ free of minimal points, i.e. $\varkappa^2 - k > 0$  at any point.
As in the classical differential geometry of surfaces in $\R^3$, the second fundamental form
determines principal tangents and principal lines. 
The surface is parametrized with respect to the principal lines if and only if $F=M=0$.

In \cite{G-M-2008}, a geometrically determined orthonormal frame field $\{x,y,b,l\}$ is introduced in a neighbourhood of each point of the surface in the following way. Assume that the surface is parametrized by the principal lines and denote
$x=\ds{\frac{z_u}{\sqrt E}}, \;\; y=\ds{\frac{z_v}{\sqrt G}}$.
Since the mean curvature vector field $H \neq 0$, we consider the unit normal vector field $b$ defined by 
$b=\displaystyle{\frac{H}{\Vert H \Vert}}$. Denote by $l$ the unit normal vector field such that
$\{x,y,b,l\}$ is a positively oriented orthonormal frame field of $M^2$.
Thus, in a neighbourhood of each point $p \in M^2$ we have a geometrically determined orthonormal frame field $\{x,y,b,l\}$, with respect to which the
following Frenet-type derivative formulas  hold true:
\begin{equation} \label{eq:Frenet-typeFormulas}
	\begin{array}{llll}
	\vspace{2mm}
		\nabla'_xx=&           & \gamma_1y  &  +\nu_1  b  \\
		\vspace{2mm}
		\nabla'_xy=&-\gamma_1x &            &  +\lambda b +\mu l \\
		\vspace{2mm}
		\nabla'_yx=&           & -\gamma_2y &  +\lambda b+\mu l \\
		\nabla'_yy=&\gamma_2x  &            &+\nu_2b
	\end{array}
	\qquad\qquad
	\begin{array}{lllll}
	\vspace{2mm}
		\nabla'_xb=& -\nu_1x    & -\lambda y  &            & +\beta_1 l  \\
		\vspace{2mm}
		\nabla'_yb=& -\lambda x & -\nu_2 y    &            & +\beta_2 l \\
		\vspace{2mm}
		\nabla'_xl=&            & -\mu y      &  -\beta_1b &            \\
		\nabla'_yl=& -\mu x     &             &  -\beta_2b &
	\end{array}
\end{equation}
where $\nu_1,\nu_2,\gamma_1,\gamma_2,\beta_1,\beta_2,\lambda,\mu$ are smooth functions determined by the geometric frame field $\{x,y,b,l\}$. We call these functions \textit{geometric functions} of the surface. Then, it can easily be obtained (see \cite{G-M-2008}) that
$$
	k=-4\nu_1\nu_2\mu^2\,, \qquad \varkappa=(\nu_1-\nu_2)\mu\,, \qquad K=\nu_1\nu_2-(\lambda^2+\mu^2) \,,
$$
$K$ being the Gauss curvature of $M^2$. The functions $\gamma_1$ and $\gamma_2$  are expressed in terms of the coefficients of the first fundamental form by the formulas
\begin{equation} \label{eq:gamma}
	\gamma_1=-\frac{(\sqrt E)_v}{\sqrt E\sqrt G},  \qquad \gamma_2=-\frac{(\sqrt G)_u}{\sqrt E\sqrt G},
\end{equation} 
and the length of the mean curvature vector $H$ is 
$$
	\| H \|=\frac{\sqrt{\varkappa^2-k}}{2|\mu|} \ .
$$
Note that $\mu \neq 0$ for surfaces free of minimal points.

The functions $\nu_1,\nu_2,\gamma_1,\gamma_2,\beta_1,\beta_2,\lambda,\mu$ satisfy the following conditions (see (5.3) in \cite{G-M-2008} or (9) in \cite{G-M-2010}):
\begin{equation} \label{eq:BasicSystem}
	\begin{array}{l}
		\ds 2\mu\gamma_2+\nu_1\beta_2-\lambda\beta_1=\frac{\mu_u}{\sqrt E} \ ; \\
	  \ds	2\mu\gamma_1-\lambda\beta_2+\nu_2\beta_1=\frac{\mu_v}{\sqrt G}\ ; \\
	  \ds 2\lambda\gamma_2+\mu\beta_1-(\nu_1-\nu_2)\gamma_1=\frac{\lambda_u}{\sqrt E}-\frac{(\nu_1)_v}{\sqrt G} \ ; \\
	  \ds 2\lambda\gamma_1+\mu\beta_2+(\nu_1-\nu_2)\gamma_2=-\frac{(\nu_2)_u}{\sqrt E}+\frac{\lambda_v}{\sqrt G} \ ; \\
	  \ds \nu_1\nu_2-(\lambda^2+\mu^2)=\frac{(\gamma_2)_u}{\sqrt E}+\frac{(\gamma_1)_v}{\sqrt G}-\big( (\gamma_1)^2+(\gamma_2)^2) \ ; \\
	  \ds \gamma_1\beta_1-\gamma_2\beta_2+(\nu_1-\nu_2)\mu=-\frac{(\beta_2)_u}{\sqrt E}+\frac{(\beta_1)_v}{\sqrt G} \ .
	\end{array}
\end{equation}

Moreover, it is proved that  these eight functions $\nu_1,\nu_2,\gamma_1,\gamma_2,\beta_1,\beta_2,\lambda,\mu$ under some natural conditions determine the surface up
to a motion in $\R^4$ (see \cite{G-M-2010}, Theorem 4.1). This is the so-called  Bonnet-type Fundamental theorem  giving the number of functions and the  number of differential equations determining the surface up to a motion in the general case.

In the next section, we will introduce special principal parameters, which we call canonical. This allows us to reduce up to four the number of functions  and the  number of PDEs determining the surface.

%============================================================================

%  Canonical principal parameters

%=============================================================================

\setcounter{equation}{0}

\section{Canonical principal parameters }

We suppose that the surface is nowhere minimal (i.e. $H$ does not vanish) and is parametrized by principal parameters.

Using (\ref{eq:gamma}) and
(\ref{eq:BasicSystem}) we obtain the following equalities:
\begin{equation} \label{eq:3.1}
	\begin{array}{l}
	\vspace{2mm}
		\ds -2\mu\frac{(\sqrt{ G})_u}{\sqrt E\sqrt G}+\nu_1\beta_2-\lambda\beta_1=\frac{\mu_u}{\sqrt E} \ ; \\
		\vspace{2mm}
	  \ds	-2\mu\frac{(\sqrt{ E})_v}{\sqrt E\sqrt G}-\lambda\beta_2+\nu_2\beta_1=\frac{\mu_v}{\sqrt G} \ ;\\
		\vspace{2mm}
	  \ds -2\lambda\frac{(\sqrt{ G})_u}{\sqrt E\sqrt G}+\mu\beta_1+(\nu_1-\nu_2)\frac{(\sqrt{ E})_v}{\sqrt E\sqrt G}=\frac{\lambda_u}{\sqrt E}-\frac{(\nu_1)_v}{\sqrt G} \ ;\\
		\vspace{2mm}
	  \ds -2\lambda\frac{(\sqrt{ E})_v}{\sqrt E\sqrt G}+\mu\beta_2-(\nu_1-\nu_2)\frac{(\sqrt{ G})_u}{\sqrt E\sqrt G}=-\frac{(\nu_2)_u}{\sqrt E}+\frac{\lambda_v}{\sqrt G} \ .
	\end{array}
\end{equation}
From the last two equations of \eqref{eq:3.1} we get
\begin{equation*} \label{eq:mu_beta}
	\begin{array}{l}
	\vspace{2mm}
		\ds\mu\beta_1
			=2\lambda\frac{(\sqrt{ G})_u}{\sqrt E\sqrt G}-(\nu_1-\nu_2)\frac{(\sqrt{ E})_v}{\sqrt E\sqrt G}+\frac{\lambda_u}{\sqrt E}-\frac{(\nu_1)_v}{\sqrt G} \ ; \\
			\vspace{2mm}
		\ds\mu\beta_2
			=2\lambda\frac{(\sqrt{ E})_v}{\sqrt E\sqrt G}+(\nu_1-\nu_2)\frac{(\sqrt{ G})_u}{\sqrt E\sqrt G}-\frac{(\nu_2)_u}{\sqrt E}+\frac{\lambda_v}{\sqrt G} \ .
	\end{array}
\end{equation*}
Hence, since $\mu \neq 0$, we can express $\beta_1$ and $\beta_2$ in terms of the other functions and replace them in the first two equations of (\ref{eq:3.1}). So, we obtain
$$
	\begin{array}{l}
	\vspace{2mm}
		\ds  -\Big( 2\lambda^2+2\mu^2-\nu_1(\nu_1-\nu_2)\Big)\frac{(\sqrt{ G})_u}{\sqrt {EG}}
		+\lambda\Big( 2\nu_1+(\nu_1-\nu_2)\Big)\frac{(\sqrt{ E})_v}{\sqrt {EG}}  \\
		\vspace{2mm}
		\ds +\nu_1\left( \frac{\lambda_v}{\sqrt G}-\frac{(\nu_2)_u}{\sqrt E}\right)
		+\lambda\left(\frac{(\nu_1)_v}{\sqrt G}-\frac{\lambda_u}{\sqrt E}\right) =\mu\frac{\mu_u}{\sqrt E} \ ;
	\end{array}
$$
$$
	\begin{array}{l}
	\vspace{2mm}
	  \ds	\lambda\Big(2\nu_2-(\nu_1-\nu_2)\Big)\frac{(\sqrt{ G})_u}{\sqrt {EG}}
	  -\Big(2\lambda^2+2\mu^2+\nu_2(\nu_1-\nu_2)\Big)\frac{(\sqrt{ E})_v}{\sqrt {EG}}  \\
		\vspace{2mm}
	  \ds +\lambda\left(\frac{(\nu_2)_u}{\sqrt E}-\frac{\lambda_v}{\sqrt G}\right)
	  +\nu_2\left(\frac{\lambda_u}{\sqrt E}-\frac{(\nu_1)_v}{\sqrt G}\right)=\mu\frac{\mu_v}{\sqrt G} \ .
\end{array}
$$

In this way we derive a system of the form
\begin{equation*} \label{eq:3.3-v1}
	\begin{array}{l}
	\vspace{2mm}
	\left(\sqrt E\right)_v=f_1\sqrt E+f_2\sqrt G \ ; \\
	\vspace{2mm}
	\left(\sqrt G\right)_u=f_3\sqrt E+f_4\sqrt G \ ,
	\end{array}
\end{equation*} 
which is equivalent to 
\begin{equation} \label{eq:3.3-v2}
	\begin{array}{l}
	\vspace{2mm}
		\ds\left(\ln\sqrt E\right)_v=f_1+f_2\frac{\sqrt G}{\sqrt E} \ ; \\
		\vspace{2mm}
		\ds\left(\ln\sqrt G\right)_u=f_3\frac{\sqrt E}{\sqrt G}+f_4 \ ,
	\end{array}   
\end{equation} 
where 
\begin{equation} \label{eq:functions}
\begin{array}{l}
	\vspace{3mm}
	f_1 \!= \! \ds{-\frac{\frac12\left((\lambda^2+\mu^2)^2\right)_v-(\lambda^2)_v\nu_1\nu_2+2\mu^2\nu_2(\nu_1)_v
	                   +\left(\lambda^2(\nu_1)_v-\frac12(\mu^2)_v\nu_1-\frac12(\nu_1^2)_v\nu_2\right)(\nu_1-\nu_2)}
	{4(\lambda^2+\mu^2)^2+(\lambda^2-2\mu^2-\nu_1\nu_2)(\nu_1-\nu_2)^2-4\lambda^2\nu_1\nu_2}} \ ; \\
	\vspace{3mm}
	f_2 \!=  \! \ds{\frac{\lambda_u(\lambda^2-\nu_1\nu_2)(\nu_1-\nu_2)+2\lambda_u\mu^2\nu_2+2\lambda^3(\nu_2)_u
	     +\frac12\lambda(\mu^2)_u(\nu_1-3\nu_2)+2\lambda\mu^2(\nu_2)_u-\lambda\nu_1(\nu_2^2)_u}
	{4(\lambda^2+\mu^2)^2+(\lambda^2-2\mu^2-\nu_1\nu_2)(\nu_1-\nu_2)^2-4\lambda^2\nu_1\nu_2}} \ ;\\
	\vspace{3mm}
	f_3 \!= \! \ds{\frac{\lambda_v(\lambda^2-\nu_1\nu_2)(\nu_2-\nu_1)+2\lambda_v\mu^2\nu_1+2\lambda^3(\nu_1)_v
	          +\frac12\lambda(\mu^2)_v(\nu_2-3\nu_1)+2\lambda\mu^2(\nu_1)_v-\lambda(\nu_1^2)_v\nu_2}
	{4(\lambda^2+\mu^2)^2+(\lambda^2-2\mu^2-\nu_1\nu_2)(\nu_1-\nu_2)^2-4\lambda^2\nu_1\nu_2} }\ ; \\
	\vspace{3mm}
	f_4 \!= \! \ds{-\frac{\frac12\left((\lambda^2+\mu^2)^2\right)_u-(\lambda^2)_u\nu_1\nu_2+2\mu^2\nu_1(\nu_2)_u
	    +\left(\lambda^2(\nu_2)_u-\frac12(\mu^2)_u\nu_2-\frac12\nu_1(\nu_2^2)_u\right)(\nu_2-\nu_1)}
	{4(\lambda^2+\mu^2)^2+(\lambda^2-2\mu^2-\nu_1\nu_2)(\nu_1-\nu_2)^2-4\lambda^2\nu_1\nu_2}} \ .
\end{array}
\end{equation}
\label{formulas fi} 

\vskip 1mm
By use of (\ref{eq:3.3-v2}) we can easily prove the following result.

\begin{lem} \label{lemma:ConstantFunctions} The functions
\begin{equation}
	\begin{array}{l}  \label{eq:semi-constantFunctions}
		\sqrt E\, e^{-\ds\int_{v_0}^v\left(f_1+f_2\frac{\sqrt G}{\sqrt E}\right)dv}  \\
		\vspace{2mm}
		\sqrt G\, e^{-\ds\int_{u_0}^u\left(f_3\frac{\sqrt E}{\sqrt G}+f_4\right)du}
	\end{array} 
\end{equation}
do not depend on $v$ and $u$, respectively.
\end{lem}

Lemma \ref{lemma:ConstantFunctions} implies that for an arbitrary point $(u_0,v_0)$ and for any constants $c_1$, $c_2$
there exist functions $\varphi(u)$ and $\psi(v)$ such that
\begin{equation} \label{eq:phi-psi}
\begin{array}{l} 
	\vspace{2mm}
		\ds\sqrt Ee^{-\ds\int_{v_0}^v\left(f_1+f_2\frac{\sqrt G}{\sqrt E}\right)dv
		                -\int_{u_0}^u\left(f_3\frac{\sqrt E}{\sqrt G}+f_4\right)(u,v_0)du+c_1}=\varphi(u) \ ;  \\
  \vspace{2mm}
		\ds\sqrt Ge^{-\ds\int_{u_0}^u\left(f_3\frac{\sqrt E}{\sqrt G}+f_4\right)du
		                -\int_{v_0}^v\left(f_1+f_2\frac{\sqrt G}{\sqrt E}\right)(u_0,v)dv+c_2}=\psi(v) \ .
	\end{array}   
\end{equation}

We give the following definition: 

\begin{defn} \label{D:def-can}
 We say that the principal parameters $(u,v)$ of a surface free of minimal points in $\R^4$ are {\it canonical principal parameters}
if the functions $\varphi(u)$ and $\psi(v)$ defined by \eqref{eq:phi-psi} are equal to 1.  
\end{defn}

\begin{lem} \label{eq:Lemma}
Suppose $\bar u=\bar u(u)$, $\bar v= \bar v(v) $ and $\bar u_0=\bar u(u_0)=u_0$, 
$\bar v_0=\bar v(v_0)=v_0$.Then
$$
	e^{\ds\int_{\bar v_0}^{\bar v}\left(\bar f_1+\bar f_2\frac{\sqrt {\bar G}}{\sqrt {\bar E}}\right)(\bar u,\bar v)d\bar v
		                +\int_{\bar u_0}^{\bar u}\left(\bar f_3\frac{\sqrt {\bar E}}{\sqrt {\bar G}}+\bar f_4\right)(\bar u,\bar v_0)d\bar u}
$$
$$	=e^{\ds\int_{v_0}^v\left(f_1+f_2\frac{\sqrt G}{\sqrt E}\right)(u,v)dv
		                +\int_{u_0}^u\left(f_3\frac{\sqrt {E}}{\sqrt {G}}+f_4\right)(u,v_0)du}
$$\end{lem}

\begin{proof}
Note that under the change of the parameters $\bar u=\bar u(u)$, $\bar v= \bar v(v) $ we have
$$\sqrt {E} = \sqrt {\bar E} \, {\bar u}_u; \quad \sqrt {G} = \sqrt {\bar G} \, {\bar v}_v,$$
 which imply that 
\begin{equation*} 
\begin{array}{ll} 
	\vspace{2mm}
		f_1 (u,v) = {\bar f_1} (\bar u(u),\bar v(v)) \, {\bar v}_v;  & \quad f_2 (u,v) = {\bar f_2} (\bar u(u),\bar v(v)) \, {\bar u}_u;  \\
  \vspace{2mm}
		f_3 (u,v) = {\bar f_3} (\bar u(u),\bar v(v)) \, {\bar v}_v;  & \quad f_4 (u,v) = {\bar f_4} (\bar u(u),\bar v(v)) \, {\bar u}_u.
	\end{array}   
\end{equation*}
Then, we obtain
$$
	\ds (f_1+f_2\frac{\sqrt {G}}{\sqrt {E}})(u,v)=(\bar f_1+\bar f_2\frac{\sqrt {\bar G}}{\sqrt {\bar E}})(\bar u(u),\bar v(v))\frac{d\bar v}{dv} \ ;
$$
$$
	\ds (f_3\frac{\sqrt {E}}{\sqrt {G}}+f_4)(u,v)=(\bar f_3\frac{\sqrt {\bar E}}{\sqrt {\bar G}}+\bar f_4)(\bar u(u),\bar v(v))\frac{d\bar u}{du} \ .
$$
Hence, 
$$
	e^{\ds\int_{v_0}^v\left(f_1+f_2\frac{\sqrt G}{\sqrt E}\right)(u,v)dv
		                +\int_{u_0}^u\left(f_3\frac{\sqrt {E}}{\sqrt {G}}+f_4\right)(u,v_0)du} =
$$
$$
	=e^{\ds\int_{v_0}^v\left(\bar f_1+\bar f_2\frac{\sqrt {\bar G}}{\sqrt {\bar E}}\right)(\bar u(u),\bar v(v))\frac{d\bar v}{dv}{dv}
		                +\int_{u_0}^u\left(\bar f_3\frac{\sqrt {\bar E}}{\sqrt {\bar G}}+\bar f_4\right)(\bar u(u),\bar v(v_0))\frac{d\bar u}{du}du} = 
$$
$$
	=e^{\ds\int_{\bar v_0}^{\bar v}\left(\bar f_1+\bar f_2\frac{\sqrt {\bar G}}{\sqrt {\bar E}}\right)(\bar u,\bar v)d\bar v
		                +\int_{\bar u_0}^{\bar u}\left(\bar f_3\frac{\sqrt {\bar E}}{\sqrt {\bar G}}+\bar f_4\right)(\bar u,\bar v_0)d\bar u} \ .
$$
\end{proof}

\begin{prop} \label{eq:ExistenceCanonicalParameters}
 Each surface free of minimal points locally admits canonical principal parameters.
\end{prop}

\begin{proof}

Given a pair of principal parameters $(u,v)$ we introduce new parameters $\bar u$, $\bar v$ by the formulas
\begin{equation*} \label{eq:defCanonParam}
	\begin{array}{l}
		\bar u=\ds\int_{u_0}^u\sqrt Ee^{-\ds\int_{v_0}^v\left(f_1+f_2\frac{\sqrt G}{\sqrt E}\right)dv
		                -\int_{u_0}^u\left(f_3\frac{\sqrt E}{\sqrt G}+f_4\right)(u,v_0)du+c_1} +u_0 \vspace{2mm} \ ;\\
		\bar v=\ds\int_{v_0}^v\sqrt Ge^{-\ds\int_{u_0}^u\left(f_3\frac{\sqrt E}{\sqrt G}+f_4\right)du
		                -\int_{v_0}^v\left(f_1+f_2\frac{\sqrt G}{\sqrt E}\right)(u_0,v)dv+c_2} +v_0 \ ,
	\end{array}
\end{equation*}  
or equivalently,
$$
	\bar u=\int_{u_0}^u \varphi(u) du +u_0 \ ;   \qquad  \bar v=\int_{v_0}^v \psi(v) dv +v_0 \,.
$$
Then  $\bar u=\bar u(u)$ and $\bar v=\bar v(v)$, so the new parameters 
$\bar u$, $\bar v$ are also principal. Moreover, $\bar u_0=\bar u(u_0)=u_0$, 
$\bar v_0=\bar v(v_0)=v_0$. Then
$$
	\bar u_u=\sqrt Ee^{-\ds\int_{v_0}^v\left(f_1+f_2\frac{\sqrt G}{\sqrt E}\right)dv
		                -\int_{u_0}^u\left(f_3\frac{\sqrt {E}}{\sqrt {G}}+f_4\right)(u,v_0)du+c_1} \ ;
$$
$$
	\bar v_v=\sqrt Ge^{-\ds\int_{u_0}^u\left(f_3\frac{\sqrt E}{\sqrt G}+f_4\right)du
		                -\int_{v_0}^v\left(f_1+f_2\frac{\sqrt G}{\sqrt E}\right)(u_0,v)dv+c_2} \ .
$$
Using Lemma \ref{eq:Lemma} we have
$$
	\sqrt{\overline E}=\sqrt{\frac{E}{(\bar u_u)^2}}=e^{\ds\int_{v_0}^v\left(f_1+f_2\frac{\sqrt G}{\sqrt E}\right)dv
		                +\int_{u_0}^u\left(f_3\frac{\sqrt {E}}{\sqrt {G}}+f_4\right)(u,v_0)du-c_1}
$$

$$
	=e^{\ds\int_{\bar v_0}^{\bar v}\left(\bar f_1+\bar f_2\frac{\sqrt {\bar G}}{\sqrt {\bar E}}\right)(\bar u,\bar v)d\bar v
		                +\int_{\bar u_0}^{\bar u}\left(\bar f_3\frac{\sqrt {\bar E}}{\sqrt {\bar G}}+ \bar f_4\right)(\bar u,\bar v_0)d\bar u-c_1}
$$
and hence $\bar\varphi (\bar u)=1$.
In a similar way, we obtain that $\bar\psi(\bar v)=1$. So, the parameters $(\bar u,\bar v)$ are canonical.

\end{proof}

\begin{lem} If $(u,v)$ and $(\bar u,\bar v)$ are canonical principal parameters
in a neighbourhood of a point $p$, then
$$
	\bar u=\pm u +u_0; \qquad\qquad	\bar v=\pm v+v_0,
$$
or
$$
	\bar u=\pm v+v_0; \qquad\qquad	\bar v=\pm u+u_0 \ .
$$
\end{lem}

\begin{proof}
Since both pairs $(u,v)$ and $(\bar u,\bar v)$ are principal parameters,  one of the following cases is possible:
$$\bar u=\bar u(u), \quad \bar v= \bar v(v);$$
or  
$$\bar u=\bar u(v), \quad \bar v= \bar v(u).$$
Suppose that $\bar u=\bar u(u)$, $\bar v= \bar v(v) $. Using Lemma \ref{eq:Lemma} we obtain
$$
	1=\bar\varphi(\bar u,\bar v)=\sqrt{\overline E} e^{-\ds\int_{\bar v_0}^{\bar v}\left(\bar f_1+\bar f_2\frac{\sqrt {\bar G}}{\sqrt {\bar E}}\right)(\bar u,\bar v)d\bar v
		                -\int_{\bar u_0}^{\bar u}\left(\bar f_3\frac{\sqrt {\bar E}}{\sqrt {\bar G}}+ \bar f_4\right)(\bar u,\bar v_0)d\bar u+c_1} =
$$
$$
	=\sqrt{\frac{E(u,v_0)}{(\bar u_u)^2}} e^{-\ds\int_{v_0}^v\left(f_1+f_2\frac{\sqrt G}{\sqrt E}\right)(u,v)dv
		                -\int_{u_0}^u\left(f_3\frac{\sqrt {E}}{\sqrt {G}}+f_4\right)(u,v_0)du+c_1} =
$$
$$
	=\frac{\varphi(u)}{|\bar u_u|}=\frac{1}{|\bar u_u|}									
$$
and hence  we get $\bar u=\pm u+u_0$. Analogously, we obtain $\bar v=\pm v+v_0$.

In a similar way, in the other case we get $\bar u=\pm v+v_0; \;\; \bar v=\pm u+u_0$. 

\end{proof}

%============================================================================

% Determining a surface by invariants

%=============================================================================

\setcounter{equation}{0}

\section{Fundamental Theorem in terms of canonical principal parameters}

According to Definition \ref{D:def-can}, in the case of canonical principal parameters $(u,v)$ we have
$$
	\ds\sqrt E=e^{\ds\int_{v_0}^v\left(f_1+f_2\frac{\sqrt G}{\sqrt E}\right)dv
		                +\int_{u_0}^u\left(f_3\frac{\sqrt E}{\sqrt G}+f_4\right)(u,v_0)du-c_1}
$$
$$
	\ds\sqrt G=e^{\ds\int_{u_0}^u\left(f_3\frac{\sqrt E}{\sqrt G}+f_4\right)du
		                +\int_{v_0}^v\left(f_1+f_2\frac{\sqrt G}{\sqrt E}\right)(u_0,v)dv-c_2}
$$
and hence, 
$$
	\ds\sqrt E(u,v_0)=e^{\ds\int_{u_0}^u\left(f_3\frac{\sqrt E}{\sqrt G}+f_4\right)(u,v_0)du-c_1}
$$
$$
	\ds\sqrt G(u,v_0)=e^{\ds\int_{u_0}^u\left(f_3\frac{\sqrt E}{\sqrt G}+f_4\right)(u,v_0)du-c_2}
$$
The last two expressions imply 
$$
	\left(\ds\sqrt {\frac{E}{G}}\right)(u,v_0)=e^{c_2-c_1}=:c= const,
$$
and hence
$$
	\ds\sqrt E(u,v_0)=e^{\ds\int_{u_0}^u\left(cf_3+f_4\right)(u,v_0)du-c_1}
$$
Analogously, we obtain
$$
	\ds\sqrt G(u_0,v)=e^{\ds\int_{v_0}^v\left(f_1+\frac 1cf_2\right)(u_0,v)dv-c_2}
$$
Denote
\begin{equation} \label{eq:g1&g2}
	\begin{array}{l}
		\vspace{2mm}
	\ds g_1(u)=e^{\ds\int_{u_0}^u\left(cf_3+f_4\right)(u,v_0)du-c_1} \\
		\vspace{2mm}
	\ds g_2(v)=e^{\ds\int_{v_0}^v\left(f_1+\frac 1cf_2\right)(u_0,v)dv-c_2}
	\end{array}
\end{equation}

Now, we shall prove the following Bonet-type fundamental theorem in terms of canonical principal parameters.

\begin{thm} \label{eq:MainTheorem}
Let $\nu_1(u,v)$, $\nu_2(u,v)$, $\lambda(u,v)$, and  $\mu(u,v)$ ($\mu \neq 0$) be smooth functions defined in a domain $\mathcal D$, ${\mathcal D} \subset \R^2$ and 
$f_i(u,v)$, $i=1,2,3,4$ be defined by formulas \eqref{eq:functions}. Let $\varphi(u,v)$, $\psi(u,v)$ be a solution to the Cauchy problem
\begin{equation*} 
	\begin{array}{l}
	\vspace{2mm}
	\ds\varphi_v=f_1\varphi+f_2\psi;  \\ 
	\vspace{2mm}
	\ds\psi_u   =f_3\varphi+f_4\psi; \\
	\end{array}
	\qquad \varphi(u,v_0)=g_1(u); \;\; \psi(u_0,v)=g_2(v),
\end{equation*} 
where $g_1(u)$ and $g_2(v)$ are defined by \eqref{eq:g1&g2}.
Suppose that the following equations are satisfied:
\begin{equation} \label{eq:4.3} 
	\begin{array}{ll}
	\vspace{2mm}
		\ds\nu_1\nu_2-(\lambda^2+\mu^2)  = & \ds-\frac1{\varphi \psi}
	           \left( \left(f_1\frac \varphi\psi+f_2\right)_v+\left(f_3+f_4\frac \psi\varphi\right)_u \right);   \\
\vspace{2mm}
\varphi\psi\,(\nu_1-\nu_2)\mu  = &
		\ds\left(2\frac{\lambda}{\mu} \frac{\psi_u}{\psi}-\frac{\nu_1-\nu_2}{\mu} \frac{\varphi_v}{\psi}+\frac{\lambda_u}{\mu}-\frac{(\nu_1)_v}{\mu} \frac{\varphi}{\psi} \right)_v - \\
		 &  \ds-  \left(2\frac{\lambda}{\mu} \frac{\varphi_v}{\varphi} + \frac{\nu_1-\nu_2}{\mu} \frac{\psi_u}{\varphi}-\frac{(\nu_2)_u}{\mu}  \frac{\psi}{\varphi}+\frac{\lambda_v}{\mu}\right)_u
	\end{array}
\end{equation}
Then, there exists a unique (up to a position in $\mathbb R^4$) surface parametrized by canonical 
principal parameters $(u,v)$ with geometric functions $\nu_1(u,v)$, $\nu_2(u,v)$, $\lambda(u,v)$, $\mu(u,v)$.
\end{thm}

\begin{proof}
Let $\nu_1(u,v)$, $\nu_2(u,v)$, $\lambda(u,v)$, and  $\mu(u,v)$ be smooth functions, $f_i(u,v)$, $i=1,2,3,4$ be determined by 
formulas \eqref{eq:functions}, and $g_1(u)$ and $g_2(v)$ be defined by \eqref{eq:g1&g2}.  Let us consider  the following Cauchy problem
\begin{equation} \label{eq:4.1}
	\begin{array}{l}
	\vspace{2mm}
	\ds\varphi_v=f_1\varphi+f_2\psi  \\ 
	\vspace{2mm}
	\ds\psi_u   =f_3\varphi+f_4\psi \\
	\end{array}
\end{equation} 
\begin{equation} \label{eq:4.1-i}
	\begin{array}{l}
	\vspace{2mm}
	\varphi(u,v_0)=g_1(u) \\
	\vspace{2mm} 
	\psi(u_0,v)=g_2(v)
	\end{array}
\end{equation} 

Note that this is an initial value problem for a canonical hyperbolic system of PDEs (see e.g. \cite{Top}) and thus system \eqref{eq:4.1} determines (at least locally) functions  $\varphi(u,v)$ and $\psi(u,v)$ 
satisfying the initial conditions \eqref{eq:4.1-i}.

 We define the functions $ E=\varphi^2 $, $ G=\psi^2 $,  and 
\begin{equation} \label{eq:3.7}
	 \ \gamma_1=-\frac{(\sqrt E)_v}{\sqrt E\sqrt G};  \qquad \gamma_2=-\frac{(\sqrt G)_u}{\sqrt E\sqrt G};
\end{equation} 
\begin{equation} \label{eq:3.7-a}
\begin{array}{l}
\vspace{2mm}
	 \ds \beta_1= \frac{1}{\mu \varphi \psi} \left(2\lambda \psi_u -(\nu_1-\nu_2) \varphi_v + \lambda_u \psi - (\nu_1)_v \varphi \right);  \\ 
	\vspace{2mm}
	\ds \beta_2= \frac{1}{\mu \varphi \psi} \left(2\lambda \varphi_v + (\nu_1-\nu_2) \psi_u -(\nu_2)_u \psi + \lambda_v \varphi  \right).
	\end{array}
\end{equation}

Now, the right hand side of the first equation of (\ref{eq:4.3}) is expressed as
$$
	-\frac1{\sqrt {EG}}\left( \left(\frac{(\sqrt E)_{v}}{\sqrt G}\right)_v +\left(\frac{(\sqrt G)_{u}}{\sqrt E}\right)_u\right)=
$$

$$
	=\frac1{\sqrt {EG}}\left( -\frac{(\sqrt E)_{vv}}{\sqrt G}+\frac{(\sqrt E)_v(\sqrt G)_v}{G}
			-\frac{(\sqrt G)_{uu}}{\sqrt E}+\frac{(\sqrt E)_u(\sqrt G)_u}{E} 
	  \right) =
$$
$$
	=\frac1{\sqrt G}(\gamma_1)_v+\frac1{\sqrt E}(\gamma_2)_u-\Big((\gamma_1)^2+(\gamma_2)^2\Big) \ .
$$
So, the first equation of (\ref{eq:4.3})  can be written as
$$
	\nu_1\nu_2-(\lambda^2+\mu^2)=\frac1{\sqrt E}(\gamma_2)_u+\frac1{\sqrt G}(\gamma_1)_v-\Big((\gamma_1)^2+(\gamma_2)^2\Big) ,
$$
which is exactly the fifth equation of  (\ref{eq:BasicSystem}).

Analogously, using the second equation of (\ref{eq:4.3})  we obtain
$$
(\nu_1-\nu_2)\mu  = \frac{(\beta_1)_v}{\sqrt G}-\gamma_1\beta_1- \frac{(\beta_2)_u}{\sqrt E}+ \gamma_2\beta_2, 
$$
which is exactly the last equation of  (\ref{eq:BasicSystem}).

On the other hand, by a straightforward calculation we can see that (\ref{eq:3.7}) and (\ref{eq:3.7-a}) imply
that the first four equations of  (\ref{eq:BasicSystem}) are fulfilled.
Consequently, system (\ref{eq:BasicSystem}) is satisfied.
Moreover, the following inequalities hold:
$$
	\frac{\mu_u}{2\mu\gamma_2+\nu_1\beta_2-\lambda\beta_1}>0 \ ;  \qquad\qquad
	\frac{\mu_v}{2\mu\gamma_1+\nu_2\beta_1-\lambda\beta_2}>0 \ .
$$
Now,  the statement follows from Theorem 4.1 in \cite{G-M-2010}.

 \end{proof}

The meaning of this theorem is that each surface in the Euclidean space $\R^4$ is determined up to a motion (a position in the space)  by four smooth functions 
$\nu_1(u,v)$, $\nu_2(u,v)$, $\lambda(u,v)$,  $\mu(u,v)$, $\mu \neq 0$, satisfying a system of PDEs. Moreover, the parameters $(u,v)$ are the canonical principal parameters of the surface.

%============================================================================

%  The case of minimal surfaces

%=============================================================================

\setcounter{equation}{0}

\section{The case of minimal surfaces}

Until now we supposed that the surface is nowhere minimal (i.e. $H$ does not vanish) 
and is parametrized by principal parameters. 
Special isothermal parameters were introduced on any minimal surface by Itoh in \cite{Itoh} and these parameters were used to prove that the  minimal non-super-conformal surfaces in  $\R^4$  are determined up to a motion by the Gaussian curvature and the normal curvature satisfying  system  \eqref{e:sys-min}.

An equivalent system satisfied by two invariant functions $\mu$ and $\nu$ is derived in  \cite{G-M-2008-minimal} on the base of the so-called canonical parameters which have the same geometrical meaning as the parameters of Itoh. 
The Frenet-type derivative formulas obtained in \cite{G-M-2008-minimal} for minimal surfaces can be obtained as a special case of formulas 
\eqref{eq:Frenet-typeFormulas} in this paper, when we take $\lambda=0$ and  $\nu=\nu_1= - \nu_2$.

The parameters $(u,v)$ of a minimal surface are said to be canonical (see \cite{G-M-2008-minimal}), if
$$
	E\sqrt{|\mu^2-\nu^2|}=1  \,; \qquad G\sqrt{|\mu^2-\nu^2|}=1 \ .
$$
On the other hand, it is easy to see that if $\lambda=0$, $\nu=\nu_1= - \nu_2$, then  the functions $f_1$, $f_2$, $f_3$, $f_4$ defined by 
 \eqref{eq:functions},  get the following form 
$$
	f_1=-\frac{(\mu^2-\nu^2)_v}{4(\mu^2-\nu^2)}\,; \qquad f_2=f_3=0 \,; \qquad  f_4=-\frac{(\mu^2-\nu^2)_u}{4(\mu^2-\nu^2)} \ .
$$
Hence, for $c_1=c_2=\frac{1}{4} (\ln |\mu^2-\nu^2|)(u_0,v_0) $ the squares of the  functions in  \eqref{eq:phi-psi} take the form
$$
	E\sqrt{|\mu^2-\nu^2|}  \,; \qquad G\sqrt{|\mu^2-\nu^2|}  \,.
$$
So, the definition of canonical parameters of minimal surfaces given in  \cite{G-M-2008-minimal} 
can be considered as a special case of our definition of canonical principal parameters in the general case of surfaces in $\R^4$.
Then, Theorem 6.4 in \cite{G-M-2008-minimal} determining the minimal surfaces in terms of two invariant functions $\nu$ and $\mu$ can be obtained as a particular case of Theorem \ref{eq:MainTheorem}.  

%============================================================================

%  The case of surfaces with parallel normalized mean curvature vector

%=============================================================================

\setcounter{equation}{0}

\section{The case of surfaces with parallel normalized mean curvature vector field}

Let $M^2$ be a surface with parallel normalized mean curvature vector field. As proved in \cite{G-M-Fil}, it follows that $\beta_1=\beta_2=0$ and hence the last equation of (\ref{eq:BasicSystem}) implies $\nu_1=\nu_2$, since $\mu \neq  0$. 
Denote this function by $\nu := \nu_1=\nu_2$. Then, formulas  (\ref{eq:3.1})  take the form
\begin{equation} \label{eq:6.1}
	\begin{array}{l}
	\vspace{2mm}
		\ds \mu_u=-2\mu\frac{(\sqrt{ G})_u}{\sqrt G}\ ; \\
		\vspace{2mm}
	  \ds	\mu_v=-2\mu\frac{(\sqrt{ E})_v}{\sqrt E}\ ; \\
		\vspace{2mm}
	  \ds \lambda_u =\nu_v\frac{\sqrt E}{\sqrt G}+\lambda\frac{\mu_u}{\mu}\ ;\\
		\vspace{2mm}
	  \ds \lambda_v =\nu_u\frac{\sqrt G}{\sqrt E}+\lambda\frac{\mu_v}{\mu}\ .
	\end{array}
\end{equation}
Using equations \eqref{eq:6.1} we obtain the following expressions for the functions $f_i$:
$$
	f_1=-\frac{\lambda^2(\mu^2)_v+\mu^2((\mu^2)_v+(\nu^2)_v)+\lambda^4(\ln\mu^2)_v+2\lambda^3\frac{\sqrt G}{\sqrt E} \nu_u
	    +\lambda(\mu^2-\nu^2) (\lambda(\ln \mu^2)_v+ 2\frac{\sqrt G}{\sqrt E} \nu_u)}{4(\lambda^2+\mu^2)^2-4\lambda^2\nu^2} \ ;
$$
$$
	f_2=\frac{2\sqrt G\lambda (\lambda^2-\nu^2)\nu_u+\mu^2(\sqrt E(\nu^2)_v+2\sqrt G\lambda \nu_u)    }
	         {4\sqrt G((\lambda^2+\mu^2)^2-\lambda^2\nu^2)} \ ;
$$
$$
	f_3=\frac{2\sqrt E\lambda (\lambda^2-\nu^2)\nu_v+\mu^2(\sqrt G(\nu^2)_u+2\sqrt E\lambda \nu_v)}
	         {4\sqrt E((\lambda^2+\mu^2)^2-\lambda^2\nu^2)} \ ;
$$
$$
	f_4=-\frac{\lambda^2(\mu^2)_u+\mu^2((\mu^2)_u+(\nu^2)_u)+\lambda^4(\ln\mu^2)_u+2\lambda^3\frac{\sqrt E}{\sqrt G} \nu_v
	    +\lambda(\mu^2-\nu^2) (\lambda(\ln \mu^2)_u+ 2\frac{\sqrt E}{\sqrt G} \nu_v)}{4(\lambda^2+\mu^2)^2-4\lambda^2\nu^2} \ .
$$
Then, we obtain the equalities
$$
	f_1+\frac{\sqrt G}{\sqrt E}f_2=-\frac{\mu_v}{2\mu} \ ;
$$
$$
	\frac{\sqrt E}{\sqrt G}f_3+f_4=-\frac{\mu_u}{2\mu} \ .
$$
Hence, if we take $c_1=\ln \sqrt{|\mu(u_0,v_0)|}$, we obtain
$$
	\varphi(u)=\sqrt E\, e^{\ds-\int_{v_0}^v\left(f_1+f_2\frac{\sqrt G}{\sqrt E}\right)dv
		                -\int_{u_0}^u\left(f_3\frac{\sqrt E}{\sqrt G}+f_4\right)(u,v_0)du+c_1}=\sqrt E\,\sqrt{|\mu|} \ .
$$
Analogously, if  $c_2=\ln \sqrt{|\mu(u_0,v_0)|}$,
then $\psi(v)=\sqrt G\,\sqrt{|\mu|}$. 

Consequently,  in the  case of surfaces with parallel normalized mean curvature vector field, our definition of canonical parameters is equivalent to  the definition given in \cite{G-M-Fil}.

%============================================================================

%  Examples

%=============================================================================

\setcounter{equation}{0}

\section{Examples }

In this section, we give examples of surfaces, demonstrating the theory and showing how we can apply Theorem  \ref{eq:MainTheorem}. Each example bellow gives  a rotational
surface in $\mathbb R^4$. 

\vskip 3mm
\begin{example} \label{ex:1} Consider the functions 
$$
	\nu_1=1\,; \qquad \nu_2=1\,; \qquad \lambda=0\,; \qquad \mu=1 \,.
$$
Using the formulas in Section 3, we obtain that the functions $f_i$  vanish identically.
Then,  system \eqref{eq:4.1} gives $\varphi(u,v)=1$, $\psi(u,v)=1$.
Following the proof of Theorem \ref{eq:MainTheorem}, we obtain
$$
	E=1 \,; \qquad G=1 \,; \qquad \beta_1=\beta_2=0 \,.
$$
Now, it is easy to see that system (\ref{eq:4.3}) is satisfied. 
According to Theorem \ref{eq:MainTheorem},  a surface in $\mathbb R^4$ is determined up to a motion. Suppose that at the initial point $z(0,0)=(0,0,0,0)$ we have the following positively 
oriented orthonormal frame: 
\begin{equation} \label{eq:initialConditions}
	\begin{array}{l}
	\vspace{2mm}
		x(0,0)=(1,0,0,0)\,;  \qquad  y(0,0)=(0,1,0,0) \,; \\
		\vspace{2mm}
		b(0,0)=(0,0,1,0)\,;  \qquad l(0,0)=(0,0,0,1) \,.
	\end{array}
\end{equation}

The derivative formulas of the surface have the following form:
\begin{equation} \label{eq:derivativeFormulas}
	\begin{array}{ll}
	\vspace{2mm}
		x_u=b              & \qquad\qquad b_u= -x \\
		\vspace{2mm}
		y_u=l & \qquad\qquad b_v= - y\\
		\vspace{2mm}
		x_v=l & \qquad\qquad l_u=- y \\
		\vspace{2mm}
		y_v=b & \qquad\qquad l_v=- x
	\end{array}
\end{equation}
First of all, this system implies the vector equations
\begin{equation} \label{eq:vectorEquations}
	x_{uu}+x=0     \,;  \qquad   x_{vv}+x=0  \ . 
\end{equation}
The first equation of \eqref{eq:vectorEquations} has the following general solution
\begin{equation} \label{eq:generalSolution}
	x(u,v) =A_1(v)\cos u+A_2(v)\sin u \ .
\end{equation}
Using the second equation of (\ref{eq:vectorEquations}), we obtain the following
ordinary differential equations:
$$
	A_1''(v)+A_1(v)=0   \,;  \qquad  A_2''(v)+A_2(v)=0 \ .
$$ 
Using (\ref{eq:initialConditions}), (\ref{eq:derivativeFormulas}) and 
(\ref{eq:generalSolution}), we determine the initial conditions of these equations:
$$
\begin{array}{ll}
\vspace{2mm}
A_1(0)=(1,0,0,0) \,; & \quad 	A_2(0)=(0,0,1,0) \,; \\
\vspace{2mm}
A_1'(0)=(0,0,0,1) \,; &\quad 	A_2'(0)=(0,-1,0,0) \ .
\end{array}
$$
So, we have the solutions:
$$
	A_1(v)=(\cos v,0,0,\sin v) \ ;  \qquad  A_2(v)=(0,-\sin v,\cos v,0) \ ,
$$
and hence
$$
	x(u,v) =(\cos u\cos v,-\sin u\sin v,\sin u\cos v,\cos u\sin v) \ .
$$ 
Since $x=z_u$, we obtain
$$
	z(u,v) =(\sin u\cos v,\cos u\sin v,-\cos u\cos v,\sin u\sin v) +w(v) \ .
$$
Hence, using $z(0,0)=(0,0,0,0)$ we obtain $w(0)=(0,0,1,0)$. We have also 
$$
	y=z_v=(-\sin u\sin v,\cos u\cos v,\cos u\sin v,\sin u\cos v)+w'(v) \ .
$$
On the other hand, by  (\ref{eq:derivativeFormulas}) 
$$
	y=-l_u=-x_{uv}=(-\sin u\sin v,\cos u\cos v,\cos u\sin v,\sin u\cos v) \ .
$$
Consequently, $w(v)=(0,0,1,0)$. Thus we get the following parametrization of the surface
\begin{equation} \label{eq:surf1}
	S: z(u,v) =(\sin u\cos v,\cos u\sin v,1-\cos u\cos v,\sin u\sin v)  \ .
\end{equation}

With this example we show how we can obtain the parametrization in terms of principal canonical parameters of a surface determined by four given functions $\nu_1$, $\nu_2$, $\lambda$, and  $\mu$ ($\mu \neq 0$).

The other invariants of the surface are
$$
	\gamma_1=\gamma_2=0 \ ;   \qquad   k=-4 \ ;  \qquad \varkappa=0 \ ;  \qquad K=0 \ .
$$
Since $\varkappa=0$, the surface $M^2$ has flat normal connection. On the other hand, since
$M^2$ is not minimal ($\varkappa^2-k\neq 0$) and $\lambda=0$, it is a non-trivial Chen surface,
see \cite {Chen}, \cite {G-V-V}.
\end{example}

\vskip 3mm
\begin{example} Let us consider the surface $M^2$ given with the parametrization
$$
	z(u,v)=(u\cos v,u\sin v, \cos v, \sin v) \ . 
$$
The coefficients of the first fundamental form are
$$
	E=1 \ ;   \qquad  F=0 \ ;  \qquad G=1+u^2  \ ,
$$ 
and it can easily be calculated that for the second fundamental tensor we have
$$
\begin{array}{ll}
\vspace{2mm}
\sigma(z_u,z_u)=(0,0,0,0) \ ; \\
\vspace{2mm}
\sigma(z_u,z_v)= \ds{\frac1{\sqrt{1+u^2}}\big( \sin v,\cos v,u\sin v,-u\cos v \big)} \ ;\\
\vspace{2mm}
\sigma(z_v,z_v)=(0,0,-\cos v,-\sin v) \ .
\end{array}
$$
The parametrization of the surface is principal and we can consider the following orthonormal pair 
$(b,l)$ of normal vector fields:
$$
	b=\sigma(z_v,z_v) \ ;    \qquad\qquad    l=\sigma(z_u,z_v) \ .
$$
The invariants of the surface are
$$
\begin{array}{llll}
\vspace{2mm}
\nu_1=0 \ ; & \qquad \nu_2=\ds{\frac1{1+u^2}} \ ; & \qquad   \lambda=0  \ ; & \qquad \mu=\ds{\frac1{1+u^2}} \ ;\\
\vspace{2mm}
\gamma_1=0 \ ; & \qquad \gamma_2=-\ds{\frac{u}{1+u^2}} \ ; & \qquad  \beta_1=0  \ ; & \qquad \beta_2=\ds{\frac u{1+u^2}}   \ ;
\end{array}
$$
$$
	k=0 \ ; \qquad \varkappa=-\frac1{(1+u^2)^2} \ ;  \qquad  K=-\frac1{(1+u^2)^2}  \ .
$$
The parameters are not canonical. To switch to canonical parameters, we can set
$$
	u=\sinh \bar u \ ;    \qquad   v=\bar v \ .
$$	

\end{example}

\vskip 3mm
\begin{example}  Let the surface $M^2$ be parametrized by 
$$
	z(u,v)=(\cosh u\cos v,\cosh u\sin v, \cos u, \sin u) \ . 
$$
Then
$$
	E=\cosh^2 u \ ;   \qquad  F=0 \ ;  \qquad G=\cosh^2 u  \ ,
$$ 
and
$$
\begin{array}{l}
\vspace{2mm}
	\sigma(z_u,z_u)= \ds{\left(\frac{\cos v}{\cosh u} , \frac{\sin v}{\cosh u}, -\cos u + \sin u \tanh u, -\sin u - \cos u \tanh u\right)} \ ; \\
\vspace{2mm}
	\sigma(z_v,z_v)= \ds{\left(-\frac{\cos v}{\cosh u} , -\frac{\sin v}{\cosh u},- \sin u \tanh u,  \cos u \tanh u\right)} \ .
	\end{array}
$$
Since $\sigma(z_u,z_u)$ and  $\sigma(z_v,z_v)$ are not parallel, the parametrization is not principal.
To find a principal parametrization, we follow a standard procedure. First, using the 
orthonormal pair of normal vectors $(a_1,a_2)$ given by
$$
	a_1=\sigma(z_v,z_v) \ ;   \qquad\qquad    a_2=\left(0 , 0,\cos u, \sin u \right)  \ ;
$$
we may derive the coefficients $c_{ij}^k$. Then we obtain the coefficients of the second fundamental form:
$$
	L=0 \ ;     \qquad\qquad    M=-\frac1{\cosh^2u} \ ;   \qquad\qquad     N=0 \ .
$$
Since $E = G$, an appropriate change of the parameters is 	
$$
	u =\bar u+\bar v \ ; \qquad\qquad   v =\bar u-\bar v  \ .
$$
With respect to the principal parameters $(\bar u,\bar v)$ we find the coefficients of the first fundamental form
$$
	\bar E=2\cosh^2(\bar u+\bar v) \ ;   \qquad \bar F=0 \ ;  \qquad \bar G=2\cosh^2(\bar u+\bar v) \ ;  
$$
and the following orthonormal pair	$(b,l)$
of normal vector fields:
$$
	 b(\bar u,\bar v) =\big(0, 0, -\cos(\bar u +\bar v), -\sin(\bar u +\bar v)\big) \ ;
$$ 
$$
	 l(\bar u,\bar v) = \left(-\frac{\cos(\bar u-\bar v)}{\cosh(\bar u+\bar v)},
	        -\frac{\sin(\bar u-\bar v)}{\cosh(\bar u+\bar v)},-\sin(\bar u+\bar v)\tanh(\bar u +\bar v),\cos(\bar u+\bar v)\tanh(\bar u+\bar v)\right) \ .
$$
Then, we obtain the invariants of the surface:
$$
	\bar\nu_1=\bar\nu_2=\frac1{2\cosh^2(\bar u+\bar v)} \ ; \qquad \bar\lambda=\frac1{2\cosh^2(\bar u+\bar v)} \ ; 
	               \qquad    \bar\mu=-\frac1{\cosh^2(\bar u+\bar v)} \ ;
$$
$$
	\bar\gamma_1=\bar\gamma_2=\bar\beta_1=\bar\beta_2=-\frac{\tanh(\bar u+\bar v)}{\sqrt 2\cosh(\bar u+\bar v)} 	  \ ;
$$
$$
	\bar k=-\frac1{\cosh^8(\bar u+\bar v)} \ ;  \qquad \bar\varkappa=0 \ ;  \qquad  \bar K=-\frac1{\cosh^4(\bar u+\bar v)}   \ .
$$
Since  $\bar \varkappa=0$, the surface is  with flat normal connection.
However, the parametrization is not canonical. A pair $(\tilde u,\tilde v)$ of canonical principal parameters
with $(u_0,v_0)=(0,0)$ is given by
$$
	\tilde u=\sqrt 2\,\bar u \ ;   \qquad   \tilde v=\sqrt 2\,\bar v    \ . 
$$

\end{example}

\vskip 3mm
\begin{example} \label{ex:4} 
Now, we consider the surface $M^2$ given with the parametrization
\begin{equation} \label{eq:surf2}
	z(u,v)=(\cos u,\sin u, \sin v, \cos v) \ . 
\end{equation}
 In this case, we obtain
$$
	\sigma(z_u,z_u)=(-\cos u,-\sin u,0,0) \ ;   \qquad\qquad
	\sigma(z_v,z_v)=(0,0,-\sin v,-\cos v) \ .
$$
Since $\sigma(z_u,z_u)$ and $\sigma(z_v,z_v)$ are not parallel, the parametrization 
is not principal. To obtain a principal parametrization we change the parameters as follows:
$$
	u =\bar u+\bar v \ ; \qquad\qquad   v =\bar u-\bar v  \ .
$$
With respect to these parameters the coefficients of the first fundamental form are 
$$
	\bar E=2 \ ;   \qquad\qquad \bar F=0 \ ;  \qquad\qquad  \bar G=2 \,.
$$
The orthonormal pair $(b,l)$ of normal vector fields is
$$
	b(\bar u, \bar v)=-\frac1{\sqrt 2}\big(\cos(\bar u+\bar v),\sin(\bar u+\bar v),\sin(\bar u-\bar v),\cos(\bar u-\bar v)\big) \ ;
$$
$$
	l(\bar u, \bar v)=\frac1{\sqrt 2}\big(-\cos(\bar u+\bar v),-\sin(\bar u+\bar v),\sin(\bar u-\bar v),\cos(\bar u-\bar v)\big) \ ;
$$
and the geometric functions of the surface are
$$
	\bar\gamma_1=\bar\gamma_2=0 \ ; \qquad \bar\nu_1=\bar\nu_2=\frac1{\sqrt 2}  \ ;
	\qquad  \bar\beta_1=\bar\beta_2=0 \ ; \qquad \bar\lambda=0  \ ; \qquad \bar\mu=\frac1{\sqrt 2} \ ;
$$
$$
	\bar k=-1 \ ;  \qquad \bar\varkappa=0 \ ; \qquad \bar K=0 \ .
$$
In particular, the surface has flat normal connection 
and is a non-trivial Chen surface as in Example \ref{ex:1}. Actually, we note that  the  functions determining the surface in this example
are proportional  to the corresponding functions in Example \ref{ex:1}.
In the following remark we comment on this.
\end{example}

\begin{remark} Let $S\ : \ z=z(u,v)$ be a surface in $\mathbb R^4$. For a given non-zero constant $\alpha$ we define
a homothetic surface  
$$\widetilde S \ :\ \tilde z(u,v) = \alpha z(u,v).$$
 Then, it can be calculated that the geometric functions 
$\tilde\nu_1,\tilde\nu_2,\tilde\gamma_1,\tilde\gamma_2,\tilde\beta_1,\tilde\beta_2,\tilde\lambda,\tilde\mu$  of $\widetilde S$ are
obtained from the corresponding functions of $S$ by dividing with $\alpha$, whereas the coefficients 
of the first fundamental forms are related by multiplying with $\alpha^2$. In particular, if we multiply by
$\sqrt 2$ the  geometric functions of the surface  in Example \ref{ex:4}, we obtain a homothetic surface $\widetilde S$ 
which have the following geometric functions:
$$
	\tilde\nu_1=\tilde\nu_2=1 \ ;	\qquad  \tilde\beta_1=\tilde\beta_2=0 \ ; \qquad \tilde\lambda=0 \ ; \qquad \tilde\mu=1 \ .
$$
They are equal to the corresponding invariant functions of the surface in  Example \ref{ex:1}.
Consequently, by Theorem \ref{eq:MainTheorem} it follows that these surfaces must coincide (up to a motion),
whereas from their parametrizations \eqref{eq:surf1} and \eqref{eq:surf2} this is not clear. However, 
there is a motion transforming $\widetilde S$ into the surface from Example \ref{ex:1}. This motion is given by 
a rotation with the following SO(4)-matrix:
$$
	\frac1{\sqrt 2}
	\left(\begin{matrix}
		0   &  1   &  1  &  0  \\
		0   &  1   &  -1 &  0  \\
		-1  &  0   &  0  &  -1  \\
		-1  &  0   &  0  &  1
	\end{matrix}\right)       
$$
and  a translation with the vector $(0,0,1,0)$.
Note that the surface with parametrization
$$
	\frac1{\sqrt 2}(\cos u,\sin u, \cos v, \sin v)  
$$
has the same invariants except for $\mu$, which has opposite sign. The reason is that this surface has 
opposite orientation.

\end{remark}

\section*{Conclusion:}
In the present paper, we introduce canonical principal parameters on the general class of surfaces in the Euclidean 4-space $\R^4$. These parameters generalize the canonical parameters introduced for minimal surfaces  and for surfaces with parallel normalized mean curvature vector field in $\R^4$. In terms of canonical parameters  any surface in $\R^4$  is determined up to a motion by four geometrically determined functions satisfying a system of natural PDEs.
The same approach can be applied to the general class of spacelike surfaces in the Minkowski 4-space $\R^4_1$.

%=============================================================================

\vskip 6mm 
\textbf{Acknowledgments:}
The  authors are partially supported by the National Science Fund, Ministry of Education and Science of Bulgaria under contract KP-06-N82/6.

\end{document}